\documentclass{amsart}

\usepackage{amssymb}
\usepackage{mathrsfs}
\usepackage{amsmath}
\usepackage{amscd}

\usepackage{graphicx}
\usepackage{subcaption}
\usepackage{wrapfig}
\usepackage{caption}

\usepackage{pinlabel}
\usepackage{enumitem}

\newtheorem{theorem}{Theorem}
\newtheorem{prop}[theorem]{Proposition}
\newtheorem{lemma}[theorem]{Lemma}

\newtheorem{definition}[theorem]{Definition}
\newtheorem{rmk}[theorem]{Remark}

\def\diam{\text{diam}}

\begin{document}

\title{Rigidity of acute triangulations of the plane}

\renewcommand{\theequation}{\arabic{section}.\arabic{subsection}.\arabic{equation}}
\numberwithin{equation}{section}
\numberwithin{theorem}{section}

\author{Tianqi Wu}
\address{Department of Mathematics, Clark University, 950 Main St, Worcester, MA 01610, USA
}
\email{mike890505@gmail.com}

\maketitle

\begin{abstract}
    We show that a uniformly acute triangulation of the plane is rigid under Luo’s discrete conformal change, extending previous results on hexagonal triangulations. Our result is a discrete analogue of the conformal rigidity of the plane. 
    We followed He's analytical approach in his work on the rigidity of disk patterns. The main tools include maximum principles, a discrete Liouville theorem, smooth and discrete extremal lengths on networks. The key step is relating the Euclidean discrete conformality to the hyperbolic discrete conformality, to obtain an $L^\infty$ bound on the discrete conformal factor.
\end{abstract}

\tableofcontents

\section{Introduction}
A fundamental property in conformal geometry is that
a conformal embedding of the plane $\mathbb R^2$ to itself must be a similar transformation.
In this paper we discretize the plane by triangulations and prove a similar rigidity result under the notion of discrete conformal change introduced by Luo \cite{luo2004combinatorial}.

Let $T=(V,E,F)$ be an (infinite) simplicial topological triangulation of the Euclidean plane $\mathbb R^2$, where $V$ is the set of vertices, $E$ is the set of edges and $F$ is the set of faces. 
Given a subcomplex $T_0=(V_0,E_0,F_0)$ of $T$, 
denote $|T_0|$ as the underlying space of $T$. 
An embedding (\emph{resp.} homeomorphism) $\phi:|T_0|\rightarrow\mathbb R^2$ is called \emph{geodesic} if $\phi$ maps each edge of $T_0$ to a geodesic arc, i.e., a straight closed line segment.
A \emph{piecewise linear metric} (\emph{PL metric} for short) on $T_0$ is represented by an edge length function $l\in\mathbb R^{E_0}_{>0}$ satisfying the triangle inequalities. 
A geodesic embedding $\phi$ of $T_0$ naturally induces a PL metric $l=l(\phi)$ on $T_0$ by letting $l_{ij}=|\phi(i)-\phi(j)|_2$. 
Luo \cite{luo2004combinatorial} introduced the following notion of discrete conformality.
\begin{definition}[Luo \cite{luo2004combinatorial}]
\label{def}
Two PL metrics $l,l'$ on $T_0=(V_0,E_0,F_0)$ are \emph{discretely conformal} if there exists some  $u\in\mathbb R^{V_0}$ such that for any edge $ij\in E_0$
$$
l'_{ij}=e^{\frac{1}{2}(u_i+u_j)}l_{ij}.
$$
In this case, $u$ is called a \emph{discrete conformal factor}, and we denote $l'=u*l$. 
\end{definition}
Given a PL metric $l$ on $T_0$, let $\theta^i_{jk}$ denote the inner angle at the vertex $i$ in the triangle $\triangle ijk$ under the metric $l$.
Then $l$ is called 

\begin{enumerate}[label=(\alph*)]
    
\item \emph{uniformly nondegenerate} if there exists a constant $\epsilon>0$ such that $\theta^i_{jk}\geq\epsilon$ for all $\triangle ijk$ in $T_0$, and

\item \emph{uniformly acute} if there exists a constant $\epsilon>0$ such that $\theta^i_{jk}\leq\pi/2-\epsilon$ for all $\triangle ijk$ in $T_0$, and

\item \emph{Delaunay} if 
$\theta^k_{ij}+\theta^{k'}_{ij}\leq\pi$ for any pair of adjacent triangles $\triangle ijk$ and $\triangle ijk'$ in $T_0$.

\end{enumerate}
A uniformly acute PL metric is clearly uniformly nondegenerate and Delaunay. 
The main result of the paper is the following.
\begin{theorem}
\label{main}
Suppose $\phi$ is a geodesic homeomorphism of $T$ and $\psi$ is a geodesic embedding of $T$. If $l(\phi),l(\psi)$ are discretely conformal and both uniformly acute, then they differ by a constant scaling.
\end{theorem}
Wu-Gu-Sun \cite{wu2015rigidity} first proved Theorem \ref{main} for the special case where $\phi(T)$ is a regular hexagonal triangulation.
Dai-Ge-Ma \cite{ dai2022rigidity} and Luo-Sun-Wu \cite{luo2020discrete} 
generalized Wu-Gu-Sun's result by allowing $l(\psi)$ to be only Delaunay rather than uniformly acute.
All these works essentially rely on the lattice structure of the embedded vertices $\phi(V)$, and apparently cannot be generalized to triangulations without translational invariance. 
To prove Theorem \ref{main}, we adopted a different approach, which is developed by He \cite{he1999rigidity} 
in his state-of-art work on the rigidity of disk patterns.

\subsection{Other Related Works}
After Luo introducing the Definition \ref{def}, various properties regarding the rigidity and convergence of the discrete conformality were discussed in \cite{bobenko2015discrete}\cite{wu2015rigidity}\cite{gu2019convergence}\cite{wu2020convergence}\cite{luo2020discrete}\cite{luo2021convergence}\cite{dai2022rigidity}.
To solve the problem of singularity in the discrete Yamabe flow, Gu et al. \cite{gu2018discrete}\cite{gu2018discrete2} proposed a revised notion of discrete conformality for piecewise Euclidean (or hyperbolic) metrics on closed surfaces with marked points, and perfectly solved the prescribed curvature problem. This major improvement in the theory of discrete conformality inspired new advanced numerical methods in computing conformal maps \cite{sun2015discrete}\cite{gillespie2021discrete}\cite{campen2021efficient}, as well as further theoretical investigations \cite{springborn2019ideal}\cite{luo2019koebe}.
Gu et al. \cite{gu2018discrete}\cite{gu2018discrete2} proposed
to use the discrete Yamabe flow to numerically compute the target metric in the prescribed curvature problem.
Since the discrete Yamabe flow may pass through different combinatorial triangulations, diagonal switches might be needed along the flow. In \cite{wu2014finiteness} it is proved that only finitely many diagonal switches are needed in a Yamabe flow. Other works on discrete geometric flows or deformations of triangle meshes could be found in \cite{zhang2014unified}\cite{ge2018combinatorial} \cite{zhu2019combinatorial}\cite{feng2020combinatorial}\cite{wu2021fractional}\cite{luo2021deformation}\cite{luo2021deformation2}\cite{luo2022deformation}\cite{luo2022spaces}.

\subsection{Notations and Conventions}
In the remaining of the paper, we will identify the plane $\mathbb R^2$ as the complex plane $\mathbb C$. 
Given $0<r<r'$, denote $D_r=\{z\in\mathbb C:|z|<r\}$
and
$A_{r,r'}=\{z\in\mathbb C:r<|z|<r'\}$.
We also denote $D=D_1$ as the unit open disk.
Given a subset $A$ of $\mathbb C$, 
$A^c$ denote the complement $\mathbb C\backslash A$ and $\partial A$ denotes the boundary of $A$ in $\mathbb C$. 
Given two subsets $A,B$ of $\mathbb C$, the diameter of $A$ is denoted by
$$
\diam(A)=\sup\{|z-z'|:z,z'\in A\},
$$
and
the distance between $A,B$ is denoted by
$$
d(A,B)=\inf\{|z-z'|:z\in A,z'\in B\}.
$$

Given a subset $V_0$ of $V$, we use the following notations and conventions.
\begin{enumerate}[label=(\alph*)]
    \item The \emph{complement} of $V_0$ is denoted as 
$V_0^c=V\backslash V_0$.

\item The \emph{boundary} of $V_0$ is denoted as
$$
\partial V_0=\{i\in V_0:\text{there exists $j\in V_0^c$ such that $ij\in E$}\}.
$$

\item The \emph{interior} of $V_0$ is denoted as
$$
int(V_0)=V_0\backslash\partial V_0=
\{i\in V_0:j\in V_0\text{ if }ij\in E\}.
$$

\item The \emph{closure} of $V_0$ is denoted as
$$
\overline {V_0}=V_0\cup\partial (V_0^c)=
(int(V_0^c))^c.
$$

\item The subcomplex generated by $V_0$ is denoted as $T(V_0)$.

\item Denote $E(V_0)=\{ij\in E:i\in  int(V_0)\text{ or }j\in  int(V_0)\}$. Notice that $E(V_0)$ generally is not the set of edges in $T(V_0)$.

\item A real-valued function on $V_0$ is often identifies as a vector in $\mathbb R^{V_0}$.
\end{enumerate}
Given $i\in V$, the \emph{1-ring neighborhood} of $i$ is the subcomplex generated by $i$ and its neighbors. In other words, the 1-ring neighborhood of $i$ is
$$
T(\{i\}\cup\{j\in V:ij\in E\}).
$$
Furthermore, we denote $R_i$ as the underlying space of the 1-ring neighborhood of $i$.
Given a subcomplex $T_0=(V_0,E_0,F_0)$ of $T$ and $l\in\mathbb R^{E_0}$ and $u\in\mathbb R^{V_0}$, if $u*l$ is a PL metric then
\begin{enumerate}[label=(\alph*)]

\item $\theta^i_{jk}(u)=\theta^i_{jk}(u,l)$ denotes 
the inner angle of $\triangle ijk$ at $i$ under $u*l$, and

\item $K_i(u)=K_i(u*l)$ denotes the discrete curvature
$$
K_i(u)=2\pi-\sum_{jk:\triangle ijk\in F}\theta^i_{jk}(u).
$$
\end{enumerate}

\subsection{Organization of the Paper}
In Section 2 we introduce necessary properties and tools for the proof of the main theorem. The proof of main Theorem \ref{main} is given in Section 3. Section 4 gives a proof of a discrete Liouville theorem, which is used in proving Theorem \ref{main}. Section 5 proves a key estimate for the discrete conformal factor by relating to the hyperbolic discrete conformality.

\subsection{Acknowledgement} 
The work is supported in part by NSF 1760471.

\section{Preparations for the Proof}
\subsection{Extremal Length and Modulus of Annuli}
We briefly review the notions of extremal length and conformal modulus. The definitions and properties discussed here are mostly well-known. One may refer \cite{ahlfors2010conformal} and \cite{lehto1973quasiconformal} for more comprehensive introductions.

A \emph{closed annulus} is a subset of $\mathbb C$ that is homeomorphic to $\{z\in\mathbb C:1\leq|z|\leq2\}$.
An \emph{(open) annulus} is the interior of a closed annulus.
Given an annulus $A$,
denote $\Gamma=\Gamma(A)$ as the set of smooth simple closed curves in $A$ separating the two boundary components of $A$. A real-valued Borel measurable function $f$ on $A$ is called \emph{admissible} if
$
\int_\gamma fds\geq1
$
for all $\gamma\in\Gamma$.
Here $ds$ denotes the element of arc length.
The \emph{(conformal) modulus} of $A$ is defined as
$$
\text{Mod}(A)=\inf\{ \int_A f^2:f\text{ is admissible}\},
$$
where $ \int_A f^2$ denotes the integral of $f(z)^2$ against the 2-dim Lebesgue measure on $A$.
From the definition it is straightforward to verify that $\text{Mod}(A)$ is conformally invariant. 
Furthermore, if $f:A\rightarrow A'$ is a $K$-quasiconformal homeomorphism between two annuli, then
$$
\frac{1}{K}\cdot \text{Mod}(A)\leq\text{Mod}(A')
\leq{K}\cdot\text{Mod}(A).
$$
Given $0<r<r'$, denote $A_{r,r'}$ as the annulus $\{z\in\mathbb C:r<|z|<r'\}$. 
It is well-known that
$$
\text{Mod}(A_{r,r'})=\frac{1}{2\pi}\log\frac{r'}{r}.
$$
Intuitively, the conformal modulus measures the relative thickness of an annulus.
If an annulus $A$ in $\mathbb C\backslash\{0\}$ contains $A_{r,r'}$, then it is ``thicker" than $A_{r,r'}$ and one can show that
$$
\text{Mod}(A)\geq\text{Mod}(A_{r,r'})
=\frac{1}{2\pi}\log\frac{r'}{r}.
$$

On the other hand, we have that
\begin{lemma}
\label{modulus}
Suppose $A\subseteq\mathbb C\backslash\{0\}$ is an annulus separating $0$ from the infinity. If $\text{Mod}(A)\geq100$, then
$A\supseteq A_{r,2r}$ for some $r>0$.
\end{lemma}

\begin{proof}
Deonte $B$ as the bounded component of $\mathbb C-A$, and $r=\max\{|z|:z\in B\}$ and 
$R=\min\{|z|:z\in\mathbb (B\cup A)^c\}$. 
If $R\geq2r$ we are done. So we may assume $R<2r$.

Then $D_{2r}\cap\gamma\neq\emptyset$ for all $\gamma\in\Gamma(A)$.
 Let $f$ be a function on $A$ such that $f(z)=1/r$ on $A\cap D_{3r}$ and $f(z)=0$ on $A\backslash D_{3r}$. If $\gamma\in\Gamma$ and 
$
\gamma\subseteq D_{3r}$, 
$$
 \int_\gamma fds= s(\gamma)\cdot\frac{1}{r}
\geq2\cdot\text{diam}(B)\cdot\frac{1}{r}
= 2r\cdot\frac{1}{r}>1.
$$ 
If $\gamma\in\Gamma$ and $\gamma\subsetneq D_{3r}$, then $\gamma$ is a connected curve connecting $D_{2r}$ and $D_{3r}^c$ and
$$
 \int_\gamma fds\geq d(D_{2r}, D_{3r}^c)\cdot\frac{1}{r}=r\cdot\frac{1}{r}=1.
$$
So $f$ is admissible and
$$
\text{Mod}(A)\leq  \int_A f^2=\frac{1}{r^2}\cdot\text{Area}(A\cap D_{3r})\leq\frac{\pi(3r)^2}{r^2}=9\pi<100.
$$
This contradicts with our assumption.
\end{proof}

\begin{rmk}
To some extend, Lemma \ref{modulus} is a consequence of Teichm\"uller's result on extremal annuli (see Theorem 4-7 in \cite{ahlfors2010conformal}). The constant $100$ is chosen for convenience and should not be optimal.
\end{rmk}

\subsection{Discrete Harmonic Functions}
Given $V_0\subset V$ and
$\eta\in\mathbb R^{E(V_0)}_{>0}$, a discrete function $f:V_0\rightarrow\mathbb R$, or equivalently a vector $f\in\mathbb R^{V_0}$, is called \emph{harmonic at} $i\in int(V_0)$ if
$$
\sum_{j:ij\in E}\eta_{ij}(f_j-f_i)=0.
$$
The following result is well-known and easy to prove.
\begin{prop}
\label{discrete laplace equation}
Suppose $V_0$ is a finite subset of $V$ and
$\eta_{ij}\in\mathbb R^{E(V_0)}_{>0}$. 
\begin{enumerate}[label=(\alph*)]

\item If $f\in\mathbb R^{V_0}$ is harmonic at $i$ for all $i\in int(V_0)$, then for all $i\in V_0$
$$
|f_i|\leq\max_{j\in\partial V_0}|f_j|.
$$

\item
Given
$g:\partial V_0 \rightarrow \mathbb R$, there exists a unique function $f:V_0\rightarrow\mathbb R$ such that

\begin{enumerate}[label=(\roman*)]
    \item $f_i=g_i$ on $\partial V_0$, and
    \item $f$ is harmonic at any $i\in int(V_0)$. 
\end{enumerate}
Furthermore, such a map $(\eta,g)\mapsto f$ is smooth from $\mathbb R^{E(V_0)}_{>0}\times\mathbb R^{\partial V_0}$ to $\mathbb R^{V_0}$.
\end{enumerate}
\end{prop}

Given $\eta\in\mathbb R^{E}_{>0}$, $f\in\mathbb R^V$ is called \emph{harmonic} if it is harmonic at all points in $V$.
It is well-known by Liouville's Theorem that any bounded smooth harmonic function on the plane is constant.
Here we have a discrete version of Liouville's Theorem.
\begin{theorem}
\label{bounded harmonic function}
Suppose $\phi$ is a geodesic embedding of $T$ and $l(\phi)$ is 
uniformly nondegenerate. Given $\eta\in\mathbb R^E_{>0}$ with $|\eta|_\infty<\infty$, then any bounded harmonic function on $(T,\eta)$ is constant. 
\end{theorem}
The proof of Theorem \ref{bounded harmonic function} 
is postponed to Section \ref{proof of the discrete Liouville theorem}.

\subsection{Differential of the Curvature Map}
The differential of $K_i(u)$ has the following elegant formula, first proposed by Luo \cite{luo2004combinatorial}.
\begin{prop}[Adapted from Theorem 2.1 in \cite{luo2004combinatorial}]
\label{differential of curvature}
Suppose
$T_0=(V_0,E_0,F_0)$ is a 1-ring neighborhood of $i\in V$ and $l\in\mathbb R^{E_0}$. Then $K_i=K_i(u)$ is a smooth function on an open set in $\mathbb R^{V_0}$, and
$$
dK_i=\sum_{j:ij\in E}\eta_{ij}(du_i-du_j).
$$
where
$\eta_{ij}=\eta_{ij}(u)$ is defined to be
\begin{equation}
\label{eta}
\eta_{ij}(u)=\frac{1}{2}\left(\cot\theta^k_{ij}(u)+\cot\theta^{k'}_{ij}(u)\right),
\end{equation}
where $\triangle ijk,\triangle ijk'$ are the two triangles in $F$ containing edge $ij$.
\end{prop}

\subsection{Maximum Principles}
We need the following maximum principle.
\begin{lemma}
\label{maximum principle}
Suppose $V_0$ is a finite subset of $V$, and $u*l,u'*l$ are Delaunay PL metrics on $T(V_0)$. If $K_i(u)=K_i(u')=0$ for all $i\in int(V_0)$, then for all $i\in V_0$
$$
|u_i'-u_i|\leq\max_{j\in \partial V_0}|u_j'-u_j|.
$$
\end{lemma}
Lemma \ref{maximum principle} is a standard consequence the following local maximum principle, which is adapted from Lemma 2.12 in \cite{ dai2022rigidity} (or Theorem 3.1 in \cite{luo2020discrete}).
\begin{lemma}
\label{local maximum principle}
Suppose $i\in V$ and $T_0=(V_0,E_0,F_0)$ is the 1-ring neighborhood of $i$ in $V$. 
Given $l\in\mathbb R^{E_0}$, if $u*l,u'*l$ are two Delaunay PL metrics on $T_0$ and $K_i(u)=K_i(u')=0$, 
then
$$
u_i'-u_i\leq \max_{j:ij\in E}(u_j'-u_j)
$$
and the equality holds if 
$(u_i'-u_i)=(u_j'-u_j)$ for any neighbor $j$ of $i$.
\end{lemma}
\begin{rmk}
Lemma 2.12 in \cite{ dai2022rigidity} is a special case of our Lemma \ref{local maximum principle}, where $u_i=u_i'=0$ is further assumed. However, by the scaling invariance these two Lemmas are really equivalent.
\end{rmk}
\subsection{Key Estimates on the Conformal Factors for Geodesic Embeddings}

\begin{lemma}
\label{compare}
Suppose $\epsilon>0$ and $\phi,\psi$ are two geodesic embeddings of a subcomplex $T_0=(V_0,E_0,F_0)$ of $T$, such that 

\begin{enumerate}[label=(\roman*)]
\item $l(\psi)=u*l(\phi)$ for some $u\in\mathbb R^{V_0}$, and

\item
the inner angles in both PL metrics $l(\phi)$ and $l(\psi)$ are at most $\pi/2-\epsilon$.
\end{enumerate}
Given $r,r'>0$ and $i\in V$, if 
$$
\phi(|T_0|)\subseteq D_{r}
$$ 
and 
$$
\psi(i)\in D_{r'/2}\subseteq D_{r'}\subseteq \psi(|T_0|),
$$ 
then 
$$
u_i\geq\log(r'/r)- M
$$
for some constant $M=M(\epsilon)>0$. 
\end{lemma}

\section{Proof of Theorem \ref{main}}
Assume 
$l(\psi)=\bar u*l(\phi)$, and all the inner angles in $l(\phi),l(\psi)$ are at most $\pi/2-\epsilon$ for a constant $\epsilon>0$. We will first prove Theorem \ref{main} assuming $\bar u:V\rightarrow\mathbb R$ is bounded in Section \ref{u is constant}, and then prove $\bar u$ is bounded in Section \ref{bounded of u}.

\subsection{Proof of Theorem \ref{main} Assuming the Boundedness of $\bar u$}
\label{u is constant}
Let us prove by contradiction and assume that $\bar u$ is not constant. Without loss of generality, 
we can do a scaling and assume 
$$
\inf_{i\in V}\bar u_i<0<\sup_{i\in V}\bar u_i
$$
and 
$$
-\inf_{i\in V}\bar u_i
=\sup_{i\in V}\bar u_i
=|\bar u|_\infty.
$$

By a standard compactness argument, it is not difficult to see that there exists a small constant $\delta=\delta(\epsilon,\bar u)\in(0,|\bar u|_\infty)$ such that if $|u|_\infty<2\delta$, 
$$
\theta^i_{jk}(u)=
\theta^i_{jk}(u,l(\phi))\geq\pi/2-\epsilon/2
$$ 
for all $\triangle ijk\in F$.
Pick a sequence of increasing subsets $V_n$ of $V$ such that  $\cup_{n=1}^\infty V_n=V$.
For each $n\in\mathbb Z_{>0}$, we will construct a smooth $\mathbb R^{V_n}$-valued function $u^{(n)}(t)=[u_i^{(n)}(t)]_{i\in V_n}$ on $(-2\delta,2\delta)$ such that
\begin{enumerate}[label=(\alph*)]


    \item $u^{(n)}(0)=0$, and

    \item  $\dot u_i^{(n)}(t)=\bar u_i/|\bar u|_\infty$ if $i\in \partial V_n$, and
    
    \item  if $i\in \text{int}(V_n)$ then
    \begin{equation}
    \label{condition c}
    \sum_{j:ij\in E}\eta_{ij}(u^{(n)}(t))
(\dot u_i^{(n)}(t)
-\dot u_j^{(n)}(t))
=0
\end{equation}
where $\eta_{ij}(u)$ is defined for all $ij\in E(V_n)$ as in equation (\ref{eta}).
\end{enumerate}
The conditions (b) and (c) give an autonomous ODE system on 
$$
\mathcal U_n=\{u\in\mathbb R^{V_n}:|u|_\infty<2\delta\}.
$$
Notice that $\eta_{ij}(u)>0$ if $u\in\mathcal U_n$. 
Then by part (b) of Lemma \ref{discrete laplace equation},
$\dot u^{(n)}(t)$ is smoothly determined by $u^{(n)}(t)$ on $\mathcal U_n$. Given the initial condition $u^{(n)}(0)=0$,
assume the maximum existence interval for this ODE system on $\mathcal U_n$ is $(t_{\min},t_{\max})$ where
$t_{\min}\in[-\infty,0)$
and $t_{\max}\in(0,\infty]$.
By the maximum principle (part (a) in Lemma \ref{discrete laplace equation}), for all $i\in V_n$
$$
|\dot u^{(n)}|_\infty\leq \max_{j\in\partial V_n}|\dot u_j^{(n)}|=
\max_{j\in\partial V_n}|\bar u_j|/|\bar u|_\infty\leq 1.
$$
So $|u^{(n)}(t)|_\infty\leq t\leq t_{\max}$ for all $t\in [0,t_{\max})$. By the maximality of $t_{\max}$, $t_{\max}=\infty$ or
$$
|u^{(n)}(t)|_\infty\rightarrow2\delta
\quad\text{ as }\quad t\rightarrow t_{\max}.
$$
So $t_{\max}\geq2\delta$ and by a similar reason $t_{\min}\leq-2\delta$. $u^{(n)}(t)$ is indeed well-defined on $(-2\delta,2\delta)$.
By Proposition \ref{differential of curvature} and equation (\ref{condition c}),
$
K_i(u^{(n)}(t))=0
$
for all $i\in  int(V_n)$.
Then by Lemma \ref{maximum principle}, for all $i\in V_n$
\begin{align}
\label{eqn32}
&|\bar u_i-u_i^{(n)}(\delta)|
\leq \max_{j\in\partial V_n}
|\bar u_j-u_j^{(n)}(\delta)|=
\max_{j\in\partial V_n}
\left(
\bar u_j-\delta\cdot
\frac{\bar u_j}{|\bar u|_\infty}
\right)
\\
\leq&(1-\frac{\delta}{|\bar u|_\infty})|\bar u|_\infty=|\bar u|_\infty-\delta.\notag
\end{align}

By picking a subsequence, we may assume that $u^{(n)}_i$ converge to $u_i^*$ on $[0,\delta]$ uniformly for all $i\in V$. 
Then $u^*=[u_i^{*}]_{i\in V}$ satisfies the following.

(a) $u^*_i(t)$ is 1-Lipschitz for all $i\in V$. As a consequence, for all $i\in V$,
$u_i^*(t)$ is differentiable at a.e. $t\in[0,\delta]$.

(b) For all $\triangle ijk\in F$,
$\theta^i_{jk}(u^*(t))\leq\frac{\pi}{2}-\frac{\epsilon}{2}$. As a consequence $\theta^i_{jk}(u^*(t))\geq\epsilon$ for all $\triangle ijk\in F$
and $\eta_{ij}(u^*(t))\leq2\cot\epsilon$ for all $ij\in E$.

(c) For all $i\in V$, $K_i(u^*(t))=0$. As a consequence for a.e. $t\in[0,\delta]$, 
$$
0=\frac{d}{dt}K_i(u^*(t))=\sum_{j:ij\in E}\eta_{ij}(u^*(t))(\dot u^*_i(t)-\dot u^*_j(t)),
$$
for all $i\in V$.

(d) By Theorem \ref{bounded harmonic function}, $\dot u^*(t)$ is constant on $V$ for a.e. $t\in[0,\delta]$. As a consequence $u_i^*(\delta)$ equals to a constant $c$ independent on $i\in V$.

(f) By equation (\ref{eqn32}),
$$
|\bar u_i -c|=|\bar u_i-u^*_i(\delta)|\leq|\bar u|_\infty-\delta
$$
for all $i\in V$. As a consequence we get the following contradiction
$$
2|\bar u|_\infty=|\sup_{i\in V}\bar u_i-\inf_{i\in V}\bar u_i|
\leq |\sup_{i\in V}\bar u_i-c|+
|\inf_{i\in V}\bar u_i-c|\leq2|\bar u|_\infty-2\delta.
$$
\subsection{Boundedness of the Conformal Factor}
\label{bounded of u}
Without loss of generality, we may assume that $\psi\circ\phi^{-1}$ is linear on each triangle $\phi(\triangle ijk)$. Then $\psi\circ\phi^{-1}$ is $K$-quasiconformal for some constant $K=K(\epsilon)>0$.
We will prove the boundedness of $\bar u$ by showing that for any $ j, j'\in V$,
$$
|\bar u_j-\bar u_{j'}|\leq 2M+2\log C+\log C'-\log2,
$$
where $M=M(\epsilon)$ is the constant given in Lemma \ref{compare} and
$C=C(\epsilon)$ is the constant given in Lemma \ref{43} and
$C'=C'(\epsilon)=e^{200\pi K}$.

Assume $j,j'\in V$. For convenience, let us assume $\phi(j)=\psi(j)=0$ by translations.
Pick $r>0$ sufficiently large such that $|\phi(j')|<r/(2C)$ and $\phi(R_j)\subseteq D_{r}$. Let 
$V_1=\{i\in V:\phi(i)\in D_{r}\}$ and 
$V_2=\{i\in V:\phi(i)\in D_{CC'r}\}$ and $T_1=T(V_1)$ and $T_2=T(V_2)$.
Then by Lemma \ref{43} we have
    \begin{equation}
    \label{31}
        \{\phi(j),\phi(j')\}\subseteq D_{r/(2C)}\subseteq D_{r/C}\subseteq\phi(|T_1|),
    \end{equation}
and
    \begin{equation*}
        \phi(|T_1|)\subseteq D_{r}
\subseteq D_{C'r}\subseteq\phi(|T_2|)
    \end{equation*}
and
    \begin{equation}
    \label{32}
        \phi(|T_2|)\subseteq D_{CC'r}.
    \end{equation}
So $A=A_{r,C'r}$ separates $\phi(|T_1|)$ and $\phi(|T_2|)^c$, and then $A'=\psi\circ\phi^{-1}(A)\ni\psi(j)=0$ separates 
$\psi(T_1)$ and $\psi(T_2)^c$.
Furthermore
$$
\text{Mod}(A')\geq\frac{1}{K}\cdot\text{Mod}(A)=\frac{1}{K}\cdot\frac{1}{2\pi}\log\frac{C'r}{r}=100.
$$
Then by Lemma \ref{modulus} there exists $r'>0$ such that
$A_{r',2r'}\subseteq A'$. So $A_{r',2r'}$ separates $\psi(T_1)$ and $\psi(T_2)^c$ and then
\begin{equation}
\label{33}
\psi(|T_1|)\subseteq D_{r'}
\end{equation}
and
\begin{equation}
\label{34}
\{\psi(j),\psi(j')\}
\subseteq D_{r'}\subseteq D_{2r'}
\subseteq\psi(|T_2|).
\end{equation}
By Lemma \ref{compare} and equations (\ref{32}) and (\ref{34}), both $\bar u_j,\bar u_{j'}$ are at least
$$
\log\frac{2r'}{CC'r}-M
=\log\frac{r'}{r}+\log\frac{2}{CC'}-M.
$$
Again by Lemma \ref{compare} and equations (\ref{33}) and (\ref{31}),
both $-\bar u_j$ and $-\bar u_{j'}$ are at least
$$
\log\frac{r/C}{r'}-M=\log\frac{r}{r'}-\log C-M.
$$
So both $\bar u_j$ and $\bar u_{j'}$ are in the interval
$$
[
\log\frac{r'}{r}+\log\frac{2}{CC'}-M
,\log\frac{r'}{r}+\log C+M
],
$$
and
$
| \bar u_j- \bar u_{j'}|
$
is bounded by the length of this interval
$$
2M+\log C-\log\frac{2}{CC'}
=2M+2\log C+\log C'-\log 2.
$$

\section{Discrete Extremal Length and the Discrete Liouville Theorem}
\label{proof of the discrete Liouville theorem}

\subsection{Electrical Networks and Discrete Extremal Length}

Discrete harmonic functions are closely related to the theory of electrical networks.
Here the 1-skeleton $(V,E)$ of the triangulation $T$ could be viewed as an electrical network, and $\eta_{ij}$ denotes the conductance of the edge $ij$, and the function $f$ denotes the electric potentials at the vertices. Then $f$ is harmonic at $i$ if and only if the outward electric flux at $i$ is $0$.
The theory of electrical networks is closely related to discrete (edge) extremal length, originally introduced by Duffin \cite{duffin1962extremal}. Here we briefly review the theory of discrete (edge) extremal length, adapted to our setting. All the definitions and properties here are well-known and one may read \cite{duffin1962extremal}\cite{he1999rigidity} for references.

Assume $V_1,V_2$ are two nonempty disjoint subsets of $V$ such that $V_0=(V_1\cup V_2)^c$ is finite.
A \emph{path} $p$ between $V_1$ and $V_2$ is a finite set of edges in 
$$
E_0=E_0(V_1,V_2)=\{ij\in E:i\in V_0\text{ or }j\in V_0\}
$$ 
such that $\gamma_p=\cup\{e:e\in p\}$ is a simple curve connecting $V_1$ and $V_2$. Denote $P=P(V_1,V_2)$ as the set of paths between $V_1$ and $V_2$.
A \emph{cut} $q$ between $V_1$ and $V_2$ is a finite set of edges in 
$
E_0
$ 
such that $q$ separates $V_1$ and $V_2$, i.e., for any path $p\in P$, $p\cap q\neq\emptyset$.
Denote $Q=Q(V_1,V_2)$ as the set of cuts between $V_1$ and $V_2$.

Given $\mu\in\mathbb R^{E}_{>0}$,
the \emph{discrete (edge) extremal length} $EL=EL(V_1,V_2,\mu)$ is defined as
\begin{equation}
\label{EL}
EL=\min\{\sum_{e\in E_0}\mu_ew_e^2:w\in\mathbb R^{E_0},\sum_{e\in q}w_e\geq1\text{ for all } q\in Q\},
\end{equation}
and
the \emph{discrete (edge) extremal width} $EW=EW(V_1,V_2,\mu)$ is defined as
$$
EW=\min\{\sum_{e\in E_0}\mu_ew_e^2:
w\in\mathbb R^{E_0},
\sum_{e\in p}\mu_ew_e\geq1\text{ for all } p\in P\}.
$$
Here $\mu_{e}$ should be viewed as the resistance of edge $e\in E$. Then the conductance of edge $e\in E$ should be $\eta_e=1/\mu_e$.
If $f:V\rightarrow\mathbb R$ is harmonic on $V_0$ with respect to $\eta$ and $f|_{V_1}=0$ and $f|_{V_2}=1$, then $w_{ij}=|f_j-f_i|/\mu_{ij}$ gives the unique minimizer in the quadratic minimization problem in equation (\ref{EL}). If we view such $f$ as an electric potential, then $w_e$ represents the current on edge $e\in E_0$ and $EL=\sum_{e\in E_0}\mu_e w_e^2$ is the electrical power in the network, which is equal to the \emph{(equivalent) resistance} between $V_1$ and $V_2$.
The discrete extremal length and width satisfy the following reciprocal theorem.
\begin{theorem}[Adapted from Corollary 1 in \cite{duffin1962extremal}]
\label{reciproal}
$EL(V_1,V_2,\mu)\cdot EW(V_1,V_2,\mu)=1$.
\end{theorem}
Now assume $\emptyset\neq V_0, V_1, V_2...$ is an increase sequence of subsets of $V$ and $\cup_{k=0}^\infty V_k=V$. Then the electric network $(T,\mu)$ is called \emph{recurrent} if
$$
EL(V_0, V_n^c,\mu)\rightarrow\infty
$$ 
as $n\rightarrow\infty$. The recurrency of the network does not depend on the choice of $V_n$'s. 
Intuitively the recurrency means that the equivalent resistance between a finite set and the infinity is infinite.
Discrete extremal length is a useful tool to prove the discrete Liouville theorem, since the recurrency implies the discrete Liouville property.
\begin{lemma}[Lemma 5.5 in \cite{he1999rigidity}]
\label{recurrent implies Liouville}
Assume $(T,\mu)$ is recurrent, and let $\eta_e=1/\mu_e$ for all $e\in E$. Then any bounded harmonic function on $(T,\eta)$ is bounded.
\end{lemma}
\subsection{Proof of the Discrete Liouville Theorem}
We need the following lemma for the proof.
\begin{lemma}
\label{43}
Suppose $\phi:|T|\rightarrow\mathbb R^2$ is a geodesic homeomorphism and any inner angle in $l(\phi)$ is at least $\epsilon>0$. Let $a\in V$ be a vertex and assume $\phi(a)=0$. Given $r>0$, denote $V_r=\{i\in V:|\phi(i)|< r\}$ and $T_r=T(V_r)$. Then there exists a constant $C=C(\epsilon)>0$ such that if
$\phi(R_a)\subseteq D_r$,
\begin{enumerate}[label=(\alph*)]
    \item 
    $
D_{r/C}\subseteq\phi(|T_r|),
$
and 
\item as a consequence $|\phi(i)|\geq r/C$ for all $i\in\partial V_r$.
\end{enumerate}
\end{lemma}
\begin{proof}
By a standard compactness argument, it is not difficult to show that there exists a constant $\delta=\delta(\epsilon)>0$ such that
for all $\triangle ijk\in F$,
$$
d( U_{ijk}^c,
\phi(\triangle ijk))
\geq\delta \cdot \diam(\phi(\triangle ijk))
$$
where 
$$
U_{ijk}= int(\phi(R_i))\cup  int(\phi(R_j))\cup  int(\phi(R_k))\supseteq \phi(\triangle ijk).
$$

We claim that $C=1+2/\delta$ is a desired constant.
Let us prove by contradiction. Suppose $r>\max\{|\phi(i)|:ai\in E\}$ and
$
D_{r/C}\not\subseteq \phi(|T_r|).
$
Then there exists $z\in D_{r/C}\backslash\phi(|T_r|)$. Since $\phi$ is a geodesic homeomorphism, there exists a triangle $\triangle ijk\in F$ such that $z\in\phi(\triangle ijk)$. Then $\triangle ijk$ is not a triangle in $T_r$ and we may assume $i\notin V_r$. 
So $|\phi(i)|\geq r$ and $ai\notin E$ and $0=\phi(a)\notin U_{ijk}$.
Then
$$
{r}/{C}\geq|0-z|
\geq d(U_{ijk}^c,\phi(\triangle ijk))
\geq\delta\cdot
\diam(\phi(\triangle ijk))
$$
$$
\geq \delta\cdot |\phi(i)-z|
\geq\delta\cdot(r-r/C)
=
(r/C)\cdot\delta(C-1)=2r/C
$$
and we get a contradiction.
\end{proof}

\begin{proof}[Proof of Theorem \ref{bounded harmonic function}]
By replacing $\eta$ by $\eta/|\eta|_\infty$ we may assume that $|\eta|_\infty=1$.
Assume $\mu\in\mathbb R^E$ is defined as $\mu_e=1/\eta_e\geq1$ for all $e\in E$. Then by Lemma \ref{recurrent implies Liouville} we only need to show that $(T,\mu)$ is recurrent. Let $\mathbf1=(1,1,...,1)\in\mathbb R^E$. Then by the definition (equation (\ref{EL})) $EL(V_1,V_2,\mu)\geq EL(V_1,V_2,\mathbf1)$ whenever well-defined.
So we only need to show that $(T,\mathbf1)$ is recurrent.

Suppose $a\in V$ is a vertex and without loss of generality we may assume that $\phi(a)=0$. Let $\epsilon>0$ be the infimum
of the inner angles in the PL metric $l(\phi)$, and $C=C(\epsilon)>1$ be the constant given in Lemma \ref{43}. 

Let 
$
r_0=\max\{|\phi(i)|:ai\in E\}
$
and
$r_n=(2C)^nr_0$
and
$V_n=\{i\in V:\phi(i)\in D_{r_n}\}$
for all $n\in\mathbb Z_{\geq0}$.
Clearly $V_n$ is an increasing sequence of subsets of $V$ and $\cup_{n=1}^\infty V_n=V$.
We will prove the recurrency of $(T,\mathbf1)$ by showing that $EL(V_0,V_n^c,\mathbf1)\rightarrow\infty$ as $n\rightarrow\infty$.

By Lemma \ref{43} (b), $|\phi(i)|\geq r_n/C=2r_{n-1}$ 
if $i\in\partial V_{n}\cup V_{n}^c=
\overline{V_{n}^c}$.
So $V_{n-1}\cap\overline{V_{n}^c}=\emptyset$, i.e., $V_{n-1}\subseteq
(\overline{V_n^c})^c
= int(V_n)$. 
It is easy to see 
\begin{equation}
\label{E0}
E_0(V_{n-1},\overline{V_n^c})\subseteq E(V_n)\backslash E(V_{n-1}).
\end{equation}
From the definition of extremal length, we have
$$
EL(V_0, V_n^c)
\geq
EL(V_0, \overline{V_1^c})
+
EL(V_1,\overline{V_2^c})
+...+
EL(V_{n-1}, \overline{V_n^c})
$$
since
\begin{enumerate}
    \item $E_0(V_0,\overline{V_1^c}),E_0(V_1,\overline{V_2^c}),...,E_0(V_{n-1},\overline{V_n^c})$ are disjoint by equation (\ref{E0}), and
    \item $Q(V_0,\overline{V_1^c}),Q(V_1,\overline{V_2^c}),...
Q(V_{n-1},\overline{V_n^c})$ are all subsets of $Q(V_0,V_{n}^c)$.
\end{enumerate}
So it suffices to show that for all $n$,
$$
EL(V_{n-1}, \overline{V_{n}^c},\mathbf1)
\geq \frac{\sin^2\epsilon}{12\pi C^2},
$$
which by Theorem \ref{reciproal} is equivalent to
$$
EW(V_{n-1},\overline{V_{n}^c},\mathbf1)
\leq \frac{12\pi C^2}{\sin^2\epsilon}.
$$
In the remaining of the proof we
denote $E_0=E_0(V_{n-1},\overline{V_n^c})$.
Pick $w_{e}=l_e/r_{n-1}$, and then for any $p\in P=P(V_{n-1},\overline{V_{n}^c})$, 
$$
\sum_{e\in P}w_e
=\frac{1}{ r_{n-1}}\sum_{e\in P}l_e
\geq\frac{1}{ r_{n-1}}\cdot
d(\phi(V_{n-1}),\phi(\overline{V_n^c}))
\geq\frac{1}{ r_{n-1}}\cdot
(2r_{n-1}-r_{n-1})=1.
$$
So
$$
EW(V_{n-1},\overline{V_n^c},\mathbf1)
\leq\sum_{e\in E_0}w_e^2
=\frac{1}{r_{n-1}^2}\sum_{e\in E_0}l_e^2
$$
and it remains to show 
$$
\sum_{e\in E_0}l_e^2\leq\frac{12\pi C^2}{\sin^2\epsilon}\cdot r_{n-1}^2.
$$
Given $e\in E$, denote $\triangle_e,\triangle_e'$ as the two triangles in $T$ containing $e$. If $e\in E_0$, then $e$ contains at least 1 vertex in $(\overline{V_n^c})^c= int(V_n)$ and
$\triangle_e,\triangle_e'$ are both triangles in $T_n$, i.e.,
$\phi(\triangle_e),\phi(\triangle_e')$ are both in $D_{r_n}$.
Given a triangle $\triangle\in F$, we denote $|\triangle|$ as the area of $\phi(\triangle)$.
Then by the sine law
$$
|\triangle ijk|=\frac{1}{2}l_{ij}l_{jk}\sin\theta^j_{ik}
\geq\frac{1}{2}l_{ij}^2
\cdot\frac{\sin\theta^i_{jk}}{\sin\theta^k_{ij}}
\cdot\sin\theta^j_{ik}
\geq l_{ij}^2\cdot\frac{\sin^2\epsilon}{2}.
$$
Notice that a triangle $\triangle\in F$ is counted for at most $3$ times in $\sum_{e\in E_0}(|\triangle_e|+|\triangle_e'|)$ and then
$$
\sum_{e\in E_0}l_e^2
\leq\frac{1}{\sin^2\epsilon}
\sum_{e\in E_0}(|\triangle_e|+|\triangle_e'|)
\leq\frac{1}{\sin^2\epsilon}
\sum_{\triangle:\phi(\triangle)\subseteq D_{r_n}}3|\triangle|
=\frac{3\pi r_n^2}{\sin^2\epsilon}
=\frac{12\pi C^2}{\sin^2\epsilon}\cdot r_{n-1}^2.
$$
\end{proof}

\section{Hyperbolic Maximum Principles and Proof of Lemma \ref{compare}}
\label{hyperbolic conformality}

Given $z_1,z_2\in D$, we denote $d_h(z_1,z_2)$ as the hyperbolic distance between $z_1,z_2$ in the Poincar\'e disk model. The (Euclidean) discrete conformal change is related with the hyperbolic discrete conformal change as follows.
\begin{lemma}
\label{lemma51}
Suppose $z_1,z_2,z_1',z_2'\in D$ and $u_1,u_2,u_1^h,u_2^h\in\mathbb R$ are such that
$$
u_i^h=u_i+\log\frac{1-|z_i|^2}{1-|z_i'|^2}
$$
for $i=1,2$. Then
$$
|z_1'-z_2'|=e^{\frac{1}{2}(u_1+u_2)}|z_1-z_2|
$$
if and only if 
\begin{equation}
\label{hyperbolic discrete conformal}
\sinh \frac{d_h(z_i',z_j')}{2}
=e^{\frac{1}{2}(u_i^h+u_j^h)}\sinh\frac{d_h(z_i,z_j)}{2}.
\end{equation}
\end{lemma}
\begin{rmk}
Equation (\ref{hyperbolic discrete conformal}) is indeed the formula of the discrete conformal change for piecewise hyperbolic metric. This formula was first proposed by Bobenko-Pinkall-Springborn \cite{bobenko2015discrete}, and $u^h_i$ in the formula is called the hyperbolic discrete conformal factor at $i$.
\end{rmk}
Lemma \ref{lemma51} could be verified by elementary computations. The proof is given in Appendix.
The hyperbolic discrete conformal factor $u^h$ also satisfies a maximum principle.

\begin{lemma}
\label{hyperbolic maximum}
Suppose $V_0$ is a subset of $V$ and $u\in\mathbb R^{V_0}$ and $\phi,\psi$ are Euclidean geodesic embeddings of $T(V_0)$, such that 
$\phi(|T(V_0)|),\psi(|T(V_0)|)\subseteq D$ and
$l(\phi),l(\psi)$ are both uniformly acute and
$l(\psi)=u*l(\phi)$. For all $i\in V_0$,
denote $z_i=\phi(i)$ and $z_i'=\psi(i)$ and
$$
u_i^h=u_i+\log\frac{1-|z_i|^2}{1-|z_i'|^2}.
$$
\begin{enumerate}[label=(\alph*)]
\item If $i\in  int(V_0)$ and $u_i^h< 0$, then there exists a neighbor $j$ of $i$ such that
$$
u_j^h< u_i^h.
$$
\item If 
$u_i^h\geq0$ for all $i\in\partial V_0$, then $u_i^h\geq0$ for all $i\in V_0$.
\end{enumerate}
\end{lemma}
We first prove Lemma \ref{compare}
using the hyperbolic maximum princple and then prove Lemma \ref{hyperbolic maximum}.
\begin{proof}[Proof of Lemma \ref{compare}]
For any $\triangle ijk\in F$, 
$$
e^{\frac{1}{2}(u_j-u_i)}=
\frac
{e^{(u_j+u_k)/2}}
{e^{(u_i+u_k)/2}}
=
\frac{l_{jk}(\psi)/l_{jk}(\phi)}
{l_{ik}(\psi)/l_{ik}(\phi)}
=
\frac
{l_{jk}(\psi)}
{l_{ik}(\psi)}
\cdot
\frac
{l_{ik}(\phi)}
{l_{jk}(\phi)}
\geq\sin^2\epsilon.
$$
So there exists a constant $C=C(\epsilon)>0$ such that $|u_j-u_i|\leq 2C$ for all $ij\in E$.
We will show that $M(\epsilon)=C(\epsilon)+3$ is a satisfactory constant.
By a scaling, we only need to prove for the special case where $r'=1$ and $r=e^{-C-2}$.

Denote
$V_1=\{i\in V:\psi(i)\in D \}$ and $z_i=\phi(i)$ and $z_i'=\psi(i)$. 
Define $u^h\in\mathbb R^{V_1}$ as
$$
u_i^h=u_i+\log\frac{1-|z_i|^2}{1-|z_i'|^2}
$$
for all $i\in V_1$.
Assume $i\in \partial V_1$, then there exists $j\in V_0-V_1$ such that $ij\in E$. 
We claim that $u_i^h\geq0$, i.e.,
$$
e^{u_i}\cdot\frac{1-|z_i|^2}{1-|z_i'|^2}\geq1.
$$
Notice that
$$
1-|z_i'|\leq|z_i'-z_j'|=e^{\frac{1}{2}(u_i+u_j)}|z_i-z_j|\leq e^{u_i+C}\cdot 2r
=2e^{-2}e^{u_i}.
$$
So
$$
e^{u_i}\cdot\frac{1-|z_i|^2}{1-|z_i'|^2}
\geq\frac{e^2}{2}\cdot\frac{1-|z_i|^2}{1+|z_i'|}\geq\frac{e^2}{2}\cdot\frac{1-r^2}{2}
\geq\frac{e^2}{2}\cdot\frac{1-(e^{-2})^2}{2}>1.
$$
By the hyperbolic maximum principle Lemma \ref{hyperbolic maximum} (b), 
$u_i^h\geq0$ for all $i\in V_1$. Then for all $i\in V_0$ with $|z_i'|<1/2$,
$$
u_i=u_i^h-\log \frac{1-|z_i|^2}{1-|z_i'|^2}
\geq-\log \frac{1-|z_i|^2}{1-|z_i'|^2}
\geq\log(1-|z_i'|^2)\geq-1=\log (r'/r)-M.
$$
\end{proof}

\subsection{Proof of the Hyperbolic Maximum Principle}
For the proof of Lemma \ref{hyperbolic maximum},
we need to briefly review the notion of hyperbolic Delaunay.
Given a subcomplex $T_0$ of $T$, an embedding $\phi:|T_0|\rightarrow D$ is called a \emph{hyperbolic geodesic embedding} if $\phi_h$ maps each edge of $T_1$ to a hyperbolic geodesic arc in $(D,d_h)$. 
Given a triangle $\triangle{ijk}$ in $T_0$ and a Euclidean or hyperbolic geodesic embedding $\phi$
of $T_0$, denote $C_{ijk}=C_{ijk}(\phi)$ as the circumcircle of $\phi(\triangle ijk)$, i.e., a round circle in the Riemann sphere $\hat{\mathbb C}$ passing through the three vertices of $\phi(\triangle ijk)$. Furthermore, we denote 
$D_{ijk}=D_{ijk}(\phi)$ as the circumdisk of $\phi(\triangle ijk)$, i.e., the closed round disk in $\hat{\mathbb C}$ such that $\partial D_{ijk}=C_{ijk}$ and $\phi(\triangle ijk)\subseteq D_{ijk}$. 
For a Euclidean geodesic embedding $\phi$ of $T_0$, it is well-known that $l(\phi)$ is Delaunay if and only if that for any pair of adjacent triangles $\triangle ijk,\triangle ijk'$ in $T_0$, 
$$
\phi(k')\notin  int(D_{ijk}).
$$ 
So here we naturally call a Euclidean or hyperbolic geodesic embedding $\phi$ \emph{Delaunay} if
$$
\phi(k')\notin  int(D_{ijk})
$$ 
for any pair of adjacent triangles 
$
\triangle ijk,\triangle ijk'
$
in $T_0$.

\begin{proof}[Proof of Lemma \ref{hyperbolic maximum} (a)]
Assume $i\in  int(V_0)$ and $T_1=(V_1,E_1,F_1)$ is the 1-ring neighborhood of $i$. Then by Lemma \ref{no overlap} below there exists a hyperbolic Delaunay geodesic embedding  $\phi_h$ (\emph{resp.} $\psi_h$) of $T_1$ such that
$\phi_h(j)=z_j$ (\emph{resp.} $\psi_h(j)=z_j'$) for all $j\in V_1$.
By Lemma \ref{lemma51},
$$
\sinh \frac{d_h(z_j',z_k')}{2}
=e^{\frac{1}{2}(u_j^h+u_k^h)}\sinh\frac{d_h(z_j,z_k)}{2}
$$
for all $jk\in E_1$.
Suppose $f_1,f_2:D\rightarrow D$ are hyperbolic isometries such that $f_1(z_i)=0$ and $f_2(z_i')=0$. 
Then $\tilde\phi_h=f_1\circ\phi_h$
(\emph{resp.} $\tilde\psi_h=f_2\circ\psi_h$)
is a hyperbolic Delaunay geodesic embedding.
Denote $\tilde z_j=\tilde\phi_h(j)$
(\emph{resp.} $\tilde z_j'=\tilde\psi_h(j)$) for all $j\in V_1$. 
Then $z_i=z_i'=0$ and
$$
\sinh \frac{d_h(\tilde z_j',\tilde z_k')}{2}
=e^{\frac{1}{2}(u_j^h+u_k^h)}\sinh\frac{d_h(\tilde z_j,\tilde z_k)}{2}
$$
for all $jk\in E_1$.
It is not hard to see that there exists a Euclidean Delaunay geodesic embedding 
$\tilde\phi$ (\emph{resp.} $\tilde\psi$) of $T_1$
such that $\tilde\phi(j)=\tilde z_j$
(\emph{resp.} $\tilde\psi(j)=\tilde z_j'$). By Lemma \ref{lemma51} $l(\tilde\psi)=\tilde u*l(\tilde\phi)$ where
$$
\tilde u_j=u_j^h-\log\frac{1-|\tilde z_j|^2}{1-|\tilde z_j'|^2}.
$$
By the Euclidean maximum principle 
Lemma \ref{local maximum principle},
$\tilde u_j\leq \tilde u_i<0$ for some neighbor $j$ of $i$. Then 
$$
|\tilde z_j'|=l_{ij}(\tilde \psi)=e^{\frac{1}{2}(\tilde u_i+\tilde u_j)}l_{ij}(\tilde \phi)=e^{\frac{1}{2}(\tilde u_i+\tilde u_j)}|\tilde z_j|<|\tilde z_j|
$$
and
$$
u_j^h=\tilde u_j+\log\frac{1-|\tilde z_j|^2}{1-|\tilde z_j'|^2}
<\tilde u_j\leq \tilde u_i
= u_i^h-\frac{1-|\tilde z_i|^2}{1-|\tilde z_i'|^2}
=u_i^h.
$$
\end{proof}

\begin{proof}[Proof of Lemma \ref{hyperbolic maximum} (b)]
If not, assume $u_i^h=\min_{j:j\in V_0}u_j^h<0$ and then $i\in int(V_0)$.
By the minimality of $u_i^h$, $u_j^h\geq u_i^h$ for any neighbor $j$ of $i$.
This contradicts with part (a).
\end{proof}
\begin{lemma}
\label{no overlap}
Suppose $i\in V$ and $T_1=(V_1,E_1,F_1)$ is a 1-ring neighborhood of $i$. If $\phi$ is a geodesic embedding of $T_1$ 
such that $\phi(|T_1|)\subseteq D$ and $l(\phi)$ is uniformly acute, then there exists a hyperbolic geodesic embedding $\phi_h$ of $T_1$ such that
$\phi_h(j)=\phi(j)$ for all $j\in V_1$. Furthermore, such $\phi_h$ is Delaunay.

\end{lemma}
\begin{proof}
Let 
$j_1,j_2,...,j_m$ be the neighbors of $i$ listed counterclockwise in $\phi(|T_1|)$. 
Denote
$z_0=\phi(i)$
and 
$z_k=\phi(j_k)$ for $k=1,...,m$.
If $\gamma(t):[0,1]\rightarrow D$ is a smooth curve such that $\gamma(0)=z_0$, then
$\dot\gamma(t)$ could be viewed as not only a complex number but also a vector in the tangent space $T_{z_0}D$ of $(D,d_h)$ at $z_0$. By this way we naturally identify $T_{z_0}D$ with $\mathbb C$.

Given $z\in D$, let
$
v(z)=\exp_{z_0}^{-1}z\in T_{z_0}D=\mathbb C
$ 
where $\exp_{z_0}:T_{z_0}D\rightarrow D$ is the exponential map at $z_0$ on the hyperbolic plane $D$. We first show that $v(z_1),...,v(z_m)$
are counterclockwise around 0 and wrap around $0$ once. More specifically, we will show that
\begin{equation}
\label{counterclockwise}
\arg\left(\frac{v({z_{k+1}})}{v({z_{k})}}\right)\in(0,\pi)
\end{equation}
and
\begin{equation}
\label{sum}
\sum_{k=1}^m
\arg\left(\frac{v({z_{k+1}})}{v({z_{k})}}\right)=2\pi
\end{equation}
where
$z_{m+1}=z_1$ and
$\arg(z)$ denotes the argument of $z$.

Assume $k\in\{1,...,m\}$. Denote $\gamma$ (\emph{resp.} $\gamma_h$) as the Euclidean straight line in $\mathbb C$ (hyperbolic geodesic in $D$) containing $z_0,z_k$. 
Then 
$\gamma$ (\emph{resp.} $\gamma_h$) cuts $\mathbb C$ (\emph{resp.} D) into two open subsets $P,P'$ (\emph{resp.} $P_h,P'_h$). 
We may assume 
$$
P=\{z\in \mathbb C:\arg\left(\frac{z-z_0}{z_k-z_0}\right)\in(0,\pi)\}
$$ 
and
$$
P_h=\{z\in  D:\arg\left(\frac{v(z)}{v(z_k)}\right)\in(0,\pi)\}.
$$ 
Then $z_{k+1}\in P$. If $\gamma_h$ is a straight line, $P_h=P\ni z_{k+1}$ and 
we have proved equation (\ref{counterclockwise}). 
If $\gamma_h$ is a round circular arc  orthogonal to $\{|z|=1\}$, there are two different cases.

Case 1: assume $z_0,z_k$ are counterclockwise on $\gamma_h$ (see Figure 1 (A)).  
If $z_{k+1}\in P\backslash P_h$, $\angle z_0z_{k}z_{k+1}>\pi/2$ 
or
$\angle z_kz_{0}z_{k+1}>\pi/2$
and it is contradictory to the acuteness assumption. So $z_{k+1}\in P_{h}$.

Case 2: assume $z_0,z_k$ are clockwise on $\gamma_h$ (see Figure 1 (B)). 
If $z_{k+1}\in P\backslash P_h$, $\angle z_0z_{k+1}z_k>\pi/2$ and it is contradictory to the acuteness assumption.
So $z_{k+1}\in P_h$.

\begin{figure}[h]
	 	\centering
	 \begin{subfigure}[h]{0.4\textwidth}
	 	 	\centering
	 	\includegraphics[width=1\textwidth]{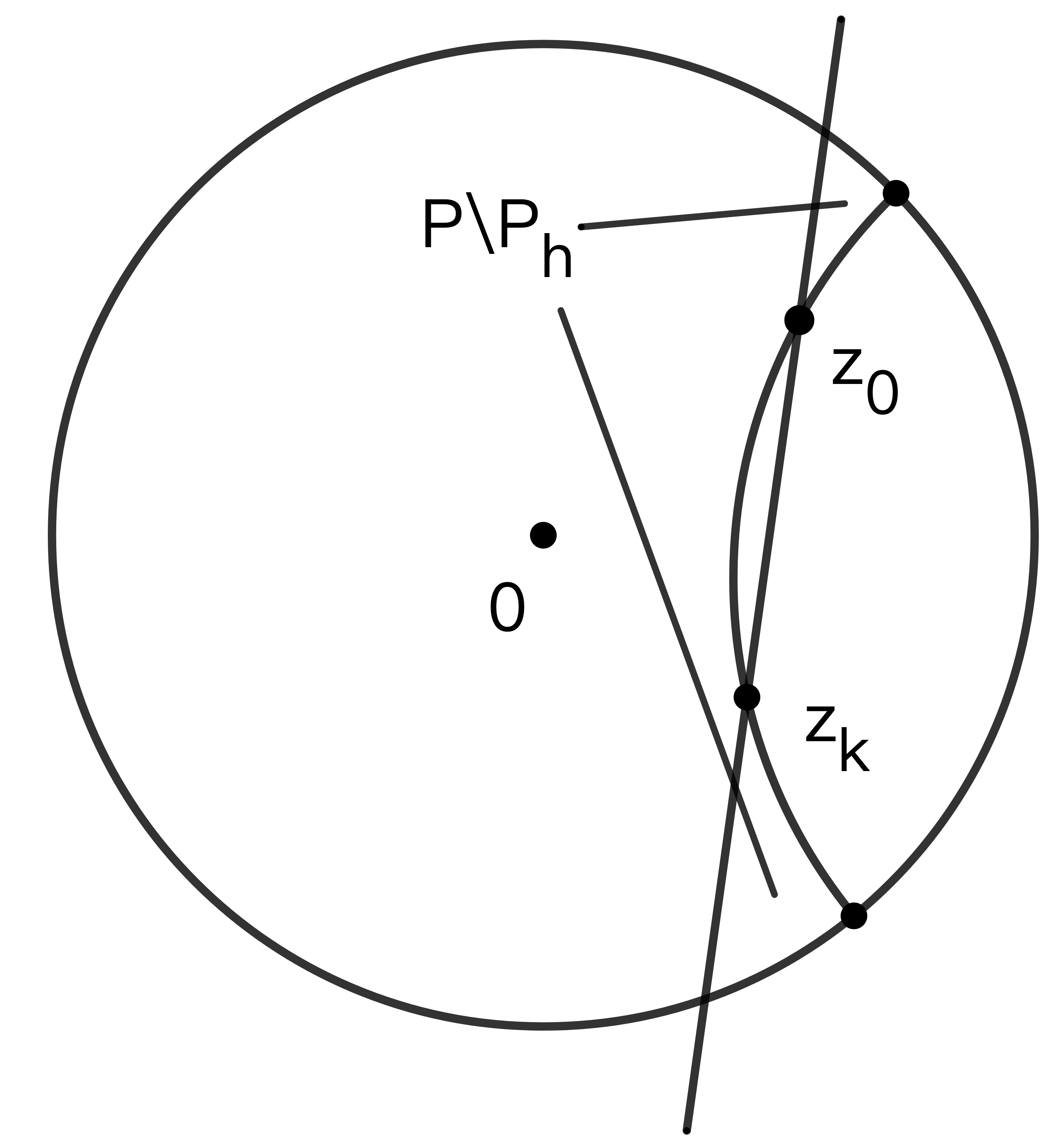}
	 	    \caption{Case 1}
	 \end{subfigure}
 \hspace{0.1cm}
    \begin{subfigure}[h]{0.4\textwidth}
	\centering
	\includegraphics[width=1\textwidth]{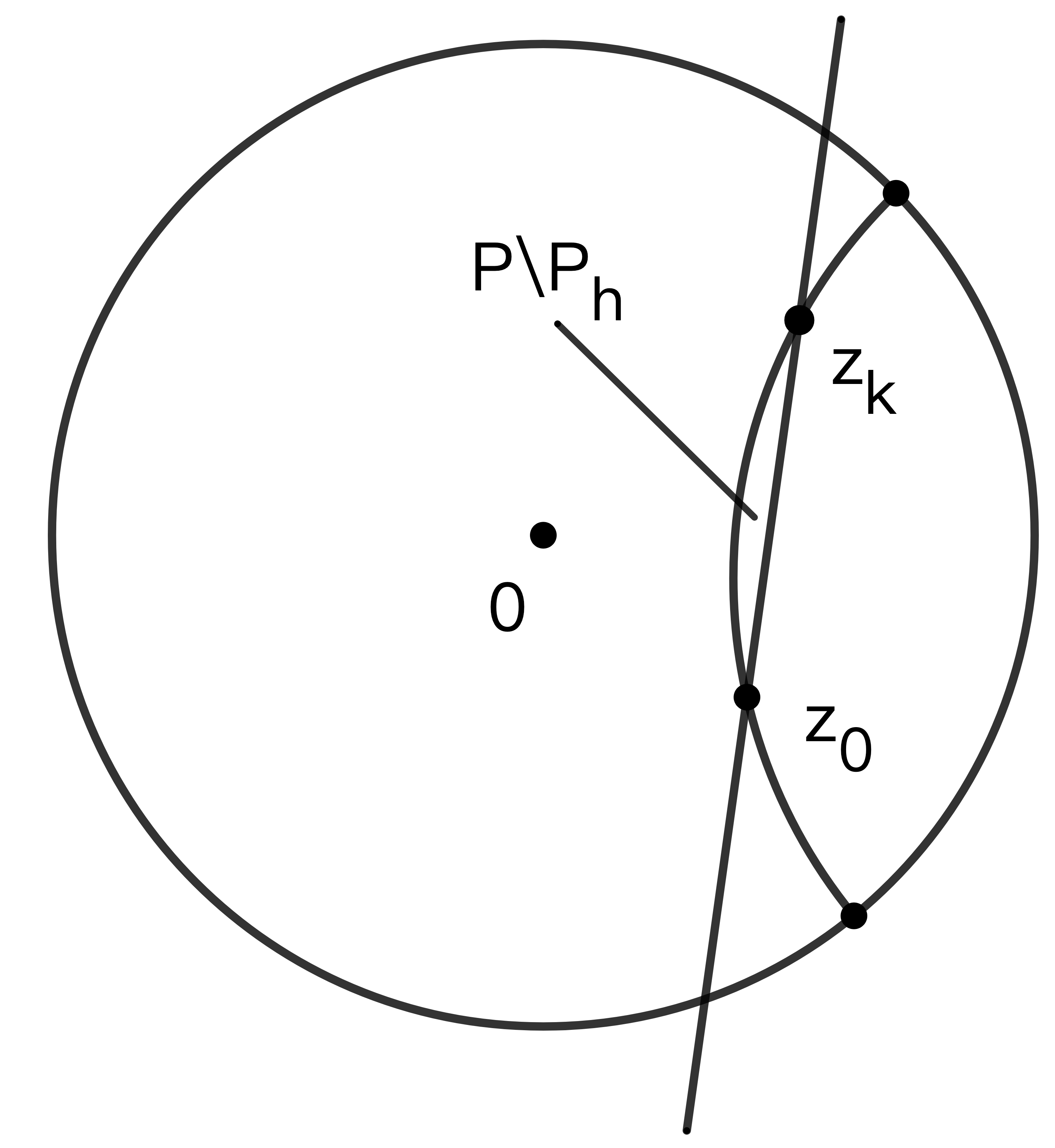}
    \caption{Case 2}
    \end{subfigure}
    \caption{}
\label{Equivalence between Delaunay and convexity}
\end{figure}

So we proved equation (\ref{counterclockwise}) and
now prove equation (\ref{sum}).
It is easy to see that

$$
\arg\left(
\frac{v(z_k)}{z_k-z_0}
\right)
\in(-\frac{\pi}{2},\frac{\pi}{2}).
$$
for all $k=1,...,m$.
We claim that
\begin{equation}
\label{arg}
\arg\left(\frac{v({z_{k+1}})}{v({z_{k})}}\right)
+
\arg\left(
\frac{v(z_k)}{z_k-z_0}
\right)
=
\arg
\left(
\frac{z_{k+1}-z_0}{z_k-z_0}\right)
+
\arg\left(
\frac{v(z_{k+1})}{z_{k+1}-z_0}
\right).
\end{equation}
Since
$$
\exp(\sqrt{-1}\cdot LHS)
=\exp(\sqrt{-1}\cdot RHS)
=\frac{v(z_{k+1})}{z_k-z_0},
$$
we have that
$$
LHS=RHS+2n\pi
$$
for some integer $n$.
On the other hand $LHS$ and $RHS$ are both bounded in
$$
(0-\frac{\pi}{2},\pi+\frac{\pi}{2})
=(-\frac{\pi}{2},\frac{3\pi}{2}),
$$
so $LHS=RHS$.
Now by adding up equation (\ref{arg}) for $k=1,...,m$ we have that
$$
\sum_{k=1}^m
\arg\left(\frac{v({z_{k+1}})}{v({z_{k})}}\right)
=
\sum_{k=1}^m
\arg
\left(
\frac{z_{k+1}-z_0}{z_k-z_0}\right)=2\pi
$$
since $\phi$ is a geodesic embedding. 
So we proved equations (\ref{counterclockwise}) and (\ref{sum}), and as a consequence there exists a hyperbolic embedding $\phi_h$ of $|T_1|$ such that $\phi_h(j)=z_j$ for all $j\in V_1$.

By equation (\ref{counterclockwise}) it is not difficult to see that 
the two circumdisks
$
D_{ij_kj_{k+1}}(\phi)
$
and
$
D_{ij_kj_{k+1}}(\phi_h)
$
are the same for $k=1,...,m$. 
So $\phi_h$ is Delaunay since $\phi$ is Delaunay.

\end{proof}

\appendix
\section{Proof of Lemma \ref{lemma51}}
\label{appendix}

\begin{proof}[Proof of Lemma \ref{lemma51}]
It suffices to show that for all $z_1,z_2\in D$,
$$
\sinh\frac{d_h(z_1,z_2)}{2}=
\frac{|z_1-z_2|}
{\sqrt {(1-|z_1|^2)(1-|z_2|^2)} }.
$$
We first consider a special case where $z_1=0$ and $z_2=r\in(0,1)$ is real.
Then
$$
d_h(z_1,z_2)=\ln\frac{1\cdot(1+r)}{1\cdot (1-r)}=\ln\frac{1+r}{1-r}
$$
and
$$
\sinh\frac{d_h(z_1,z_2)}{2}
=\frac{1}{2}\sqrt{\frac{1+r}{1-r}}-
\frac{1}{2}\sqrt{\frac{1-r}{1+r}}
=\frac{r}{\sqrt{1-r^2}}
=\frac{|z_1-z_2|}
{\sqrt {(1-|z_1|^2)(1-|z_2|^2)} }.
$$
For general $z_1,z_2\in D$, we can find a hyperbolic isometric map $f(z)= \frac{z-a}{1-\bar a z}$ such that $f(z_1)=0$  and $f(z_2)$ is a positive real number. We only need to verify that 
$$
\frac{|f(z_1)-f(z_2)|^2}{{(1-|f(z_1)|^2)(1-|f(z_2)|^2)}}
=
\frac{|z_1-z_2|^2}{{(1-|z_1|^2)(1-|z_2|^2)}}.
$$
This equality can be derived from
\begin{align*}
&\frac{|f(z_1)-f(z_2)|^2}{{(1-|f(z_1)|^2)(1-|f(z_2)|^2)}}\\
=&\frac
{\big|(z_1-a)(1-\bar a z_2)-(1-\bar a z_1)(z_2-a)\big|^2}
{
\big(|1-\bar az_1|^2-|z_1-a|^2\big)\cdot
\big(|1-\bar az_2|^2-|z_2-a|^2\big)
}\\
=&\frac
{
\big|(1-a\bar a)(z_1-z_2)\big|^2
}
{
\big(|1-\bar az_1|^2-|z_1-a|^2\big)\cdot
\big(|1-\bar az_2|^2-|z_2-a|^2\big)
}
\end{align*}
and
\begin{align*}
&|1-\bar az_1|^2-|z_1-a|^2
=(1-\bar az_1)(1-a\bar z_1)-(z_1-a)(\bar z_1-\bar a)\\
=&(1-a \bar a)(1-z_1\bar z_1)=(1-a\bar a)(1-|z_1|^2).
\end{align*}
and similarly
$$
|1-\bar az_2|^2-|z_2-a|^2=
(1-a\bar a)(1-|z_2|^2).
$$
\end{proof}

\bibliography{rigidity}
\bibliographystyle{alpha}

\end{document}